\documentclass{article}  
\usepackage{amsmath,amssymb,amsthm,
	    graphicx,
	    tabularx,
	    booktabs,
	    color,
	    xspace, amsfonts
	  }
\usepackage{array}
\newcolumntype{L}[1]{>{\raggedright\let\newline\\\arraybackslash\hspace{0pt}}m{#1}}
\newcolumntype{C}[1]{>{\centering\let\newline\\\arraybackslash\hspace{0pt}}m{#1}}
\newcolumntype{R}[1]{>{\raggedleft\let\newline\\\arraybackslash\hspace{0pt}}m{#1}}
\usepackage{caption}
\usepackage{authblk}
\usepackage{setspace}
\usepackage{textcomp}
\usepackage[mathscr]{eucal}\RequirePackage{bera}

\usepackage{indentfirst}
\usepackage{mathtools,enumerate}

\newcommand{\R}{\mathbb{R}}

\renewcommand{\vec}[1]{\boldsymbol{#1}}
\newcommand{\F}{\overline{F}}

\newcommand{\E}{\mathbb{E}}

\newcommand{\N}{\mathbb{N}}

\newcommand{\U}{\mathrm{U}}

\newcommand{\as}{\overset{\text{a.s.}}{=}}

\newcommand{\dif}{\mathrm{d}}

\newcommand{\abs}[1]{|#1|}

\newcommand{\RR}{\mathbb{R}}

\newcommand{\p}{\mathbb{P}}

\newcommand{\id}{\vec{1}}

\theoremstyle{plain}
\newtheorem{theorem}{Theorem}[section]
\newtheorem{corollary}[theorem]{Corollary}
\newtheorem{lemma}[theorem]{Lemma}
\newtheorem{proposition}[theorem]{Proposition}
\theoremstyle{definition}
 
\newtheorem{example}[theorem]{Example} 
\theoremstyle{remark}
\newtheorem{remark}[theorem]{Remark}

\begin{document}  

\title{Centers of probability measures without the mean}

\author{%
  Giovanni Puccetti\footnote{Department of Economics, Management and Quantitative Methods, University of Milano, Italy.}, \quad Pietro Rigo\footnote{Department of Mathematics, University of Pavia, Italy.},\quad
Bin Wang\footnote{Key Laboratory of Random Complex Structures and Data Science, Academy of Mathematics and Systems Science, Chinese Academy of Sciences, Beijing, China.},\quad
Ruodu Wang\footnote{Department of Statistics and Actuarial Science, University of Waterloo, Waterloo, Canada.}}

 \date{\today}
\maketitle

\begin{abstract}
In the recent years, the notion of mixability has been developed with applications to optimal transportation, quantitative finance and operations research. 
An $n$-tuple of distributions is said to be jointly mixable if there exist $n$ random variables following these distributions and adding up to a constant, called center, with probability one. When the $n$ distributions are identical, we speak of complete mixability. If each distribution has finite mean, the center is obviously the sum of the means. In this paper, we investigate the set of centers of completely and jointly mixable distributions not having a finite mean. In addition to several results, we show the (possibly counterintuitive) fact that, for each $n \geq 2$, there exist
$n$ standard Cauchy random variables adding up to a constant $C$ if and only if
$$|C|\le\frac{n\,\log (n-1)}{\pi}.$$ 

\textbf{MSC2000 subject classification}: Primary 60E05, Secondary 90B30.


   \end{abstract}
 
\section{Introduction}\label{intro}


In the recent years, the field of complete and joint mixability \cite{WW16} has been rapidly developing.
Mixability serves as a building block for the solutions of many 
optimization problems under marginal-distributional constraints. Applications are found in optimal transportation \cite{lR13}, quantitative finance \cite{EPR13,BRV14} and operations research \cite{uuH15,BKP16}. 

In this paper, we study the set of centers of completely and jointly mixable distributions. Our main result (Theorem~\ref{th:cauchy}) is that $n$ standard Cauchy random variables can add up to a constant $C$ if and only if
\begin{gather}\label{ch7y}
\abs{C}\le\frac{n\,\log (n-1)}{\pi}.
\end{gather}
Even if apparently innocuous, the proof (or at least our proof) of such a result is quite involved.

To be formal, let $\Gamma(\mu_1,\ldots,\mu_n)$ be the collection of probability measures on $\mathbb{R}^n$ having one-dimensional marginals $\mu_1,\ldots,\mu_n$, where $\mu_1,\ldots,\mu_n$ are probability measures on $\mathbb{R}$. Denote by $\mu$ the standard Cauchy distribution. Then, for any $n\ge 2$ and $C\in\R$, there exists $\lambda\in\Gamma\bigl(\mu,\dots,\mu)$ such that
\begin{gather}\label{nm9hy6}
\lambda\big(\bigl\{x\in\mathbb{R}^n:x_1+\ldots+x_n=C\bigr\}\big)=1,
\end{gather}
if and only if $C$ satisfies condition \eqref{ch7y}. In general, for any probability measure $\mu$ on $\mathbb{R}$, the set of $C\in\R$ satisfying~\eqref{nm9hy6} for some $\lambda\in\Gamma(\mu,\ldots,\mu)$ is compact (Proposition~\ref{th:metric}) and is finite in case $\mu$ is discrete (Proposition~\ref{disc}).

The existence of a probability measure $\lambda$ with given marginals and satisfying~\eqref{nm9hy6} is meaningful with respect to the notions of \emph{complete mixability} and \emph{joint mixability}, as introduced in \cite{WW11} and \cite{WPY13}. Let $\mu,\mu_1,\dots,\mu_n$ be probability measures on $\RR$. The $n$-tuple $(\mu_1,\dots,\mu_n)$ is said to be \emph{jointly mixable} (JM) if condition~\eqref{nm9hy6} holds for some $C\in\mathbb{R}$ and some $\lambda\in\Gamma(\mu_1,\ldots,\mu_n)$. In this case, $C$ is called a \emph{center} of $(\mu_1,\dots,\mu_n)$. Similarly, $\mu$ is $n$-\emph{completely mixable} ($n$-CM) if there exist $C\in\mathbb{R}$ and $\lambda\in\Gamma(\mu,\ldots,\mu)$ satisfying condition \eqref{nm9hy6}. In this case, $C/n$ is an $n$-center of $\mu$. Clearly, $C/n$ coincides with the mean of $\mu$ provided the latter exists and is finite. In this sense, the notion of $n$-center can be seen as a generalization of the notion of mean. The term \emph{$n$-center} is used to stress the dependence on $n$. However, the ``$n$-" notation will be dropped when clear from the context.

The historical motivation for investigating mixability was to minimize var$\bigl(\sum_{i=1}^nX_i\bigr)$, where $X_1,\ldots,X_n$ are real random variables with given marginal distributions. In fact, the idea of building random variables with constant sum, or at least whose sum has minimum variance, goes back to~\cite{GR81}, where complete mixability of the uniform distribution was shown. Random sums with minimal variance were further investigated in~\cite{RU02b}, where complete mixability of symmetric unimodal distributions was established. Complete mixability and joint mixability of distributions with monotone densities are characterized in \cite{WW11} and \cite{WW16}, respectively. From an analytical viewpoint, mixability can be seen as an extension of the concept of countermonotonicity (negative dependence) in dimensions $n \geq 3$ and further mathematical properties are collected in~\cite{PW15}.

In this paper, in addition to the results mentioned above, various other useful facts are proved. Amongst them, we mention Example~\ref{ex:01}, which provides the first (to our knowledge) explicit construction of two joint (complete) mixes having the same marginal distributions and different centers. 

A last remark (connected with our main result) is that, still today, the Cauchy distribution continues to exhibit
some rather unexpected properties; see e.g. \cite{MR3546444}.

\subsection*{Notation} Throughout this paper, $n$ is a positive integer. For any $A\subset\RR^n$, we say ``a probability measure on $A$" to mean ``a probability measure on the Borel $\sigma$-field of $A$". We write $X\sim\nu$ to mean that $\nu$ is the probability distribution of the random variable $X$ and $X\sim Y$ to mean that $X$ and $Y$ have the same law. We always denote by $\mathcal{B}$ the Borel $\sigma$-field of $\R$ and by $\mu$ (with or without indices) a probability measure on $\mathcal{B}$. Also, for any set $Z$, $\delta_z$ stands for the point mass at $z\in Z$.

For $x\in\mathbb{R}^n$ and $i=1,\ldots,n,$ the $i$-th coordinate of $x$ is denoted by $x_i$. If $\lambda$ is a probability measure on $\mathbb{R}^n$, the $i$-th one-dimensional marginal of $\lambda$ is the probability measure on $\mathbb{R}$ given by $A\mapsto\lambda\big(\bigl\{x\in\mathbb{R}^n:x_i\in A\bigr\}\big)$.

Finally, all random variables are defined on a common probability space $(\Omega,\mathcal F,\p)$.

\section{(Non)-Uniqueness of the center}\label{subse:2.2}

A \emph{joint mix} for $(\mu_1,\ldots,\mu_n)$ with center $C$ is an $n$-tuple $(X_1,\ldots,X_n)$ of real random variables such that $X_i\sim\mu_i$, $1 \leq i \leq n$, and $\sum_{i=1}^nX_i\as C$. Similarly, $(X_1,\ldots,X_n)$ is a $n$-\emph{complete mix} for $\mu$ with center $c$ if $X_i\sim\mu$,$1 \leq i \leq n$,  and $\sum_{i=1}^nX_i\as nc$.

Not all probability measures on $\R$ are completely mixable. For instance, $\mu$ is necessarily symmetric if it is 2-CM and centered at 0. Or else, $\mu$ is not $n$-CM for any $n$ if the support of $\mu$ is bounded above (below) but not below (above). For a broad list of jointly and completely mixable distributions we refer to~\cite{PW15} and the references therein. Here, we start by noting that the existence of a joint mix always delivers the existence of a complete mix (with average marginal distribution) and that a complete mix with a given marginal can always be taken to be exchangeable.

\begin{proposition}\label{bv6r}
\begin{enumerate}[(i)]
\item If $(\mu_1,\dots,\mu_n)$ is JM with center $C$, then $\mu=(\mu_1+\dots+\mu_n)/n$
is $n$-CM with center $C/n$.
\item Each $n$-CM probability measure on $\R$ admits an exchangeable $n$-complete mix.
\end{enumerate}
\end{proposition}
\begin{proof}
(i) Let $X=(X_1,\dots,X_n)$ be a joint mix for $(\mu_1,\dots,\mu_n)$ with center $C$. Define
\begin{gather*}
Z=(X_{\pi_1},\ldots,X_{\pi_n})
\end{gather*}
where $\pi=(\pi_1,\dots,\pi_n)$ is a uniform random permutation of $\{1,\ldots,n\}$ independent of $X$.
A uniform random permutation of $\{1,\ldots,n\}$ is a random permutation $\pi$ such that $\p(\pi=\sigma)=1/n!$ for each permutation $\sigma$ of $\{1,\ldots,n\}$. Then, $\sum_{i=1}^n Z_i=\sum_{i=1}^n X_{\pi_i}=\sum_{i=1}^n X_i\as C$. By independence of $X$ and $\pi$ and recalling that $X_i\sim\mu_i$, one obtains
\begin{gather*}
\p(Z_i\in A)=\sum_{\sigma\in P_n}\p(X_{\sigma_i}\in A,\,\pi=\sigma)=\frac{1}{n!}\,\sum_{\sigma\in P_n}\p(X_{\sigma_i}\in A)=\frac{1}{n}\sum_{i=1}^n\p(X_i\in A)=\mu(A),
\end{gather*}
for each $i$ and $A\in\mathcal{B}$, where $P_n$ is the set of all permutations of $\{1,\ldots,n\}$. Therefore, $\mu$ is $n$-CM with center $C/n$ and $Z$ is a joint mix for $\mu$.

(ii) Given an $n$-CM probability $\nu$ on $\R$, take $\mu_1=\dots=\mu_n=\nu$ in (i). Then, $Z$ is an exchangeable joint mix for $\nu$.
\end{proof}

The next example, even if obvious, is helpful in proving Theorem~\ref{th:cauchy} below.

\begin{example}\label{bh9ik7}
Let $\nu$ and $\gamma$ be probability measures on $\R$. Suppose $\nu$ is $k$-CM and $\gamma$ is $(n-k)$-CM, where $1\le k< n$. Define $\mu_i=\nu$ for $1\le i\le k$ and $\mu_i=\gamma$ for $k< i\le n$. Then, $(\mu_1,\ldots,\mu_n)$ is clearly JM, so that
$$\frac{k\nu+(n-k)\gamma}{n}$$
is $n$-CM by Proposition \ref{bv6r}. In particular, $\frac{k\delta_x+(n-k)\delta_y}{n}$ is $n$-CM for any $x,y\in\R$.
\end{example}
\bigskip

An intriguing question is whether the center of mixable distributions is unique. Obviously, if $(\mu_1,\dots,\mu_n)$ is JM and each $\mu_i$ has finite mean, then $(\mu_1,\dots,\mu_n)$ has a unique center $C$, namely $C=\sum_{i=1}^n\int x\,\mu_i(dx)$. Analogously, if $\mu$ is $n$-CM and has finite mean, $\int x\,\mu(dx)$ is the only center of $\mu$. Uniqueness of the center is also clear for $n=1$, since $\mu$ is 1-CM if and only if it is degenerate.

In view of~\cite{gS77}, if $X$ and $Y$ are real random variables such that $\E(X+Y)$ exists (finite or infinite) then $\E(X+Y)$ depends only on the marginal distributions of $X$ and $Y$, in the sense that $\E(U+V)=\E(X+Y)$ provided $U\sim X$, $V\sim Y$ and $\E(U+V)$ exists. It follows that the center is unique for $n=2$. This fact also admits an obvious direct proof: if $(X_1,X_2)$ and $(Y_1,Y_2)$ are joint mixes for $(\mu_1,\mu_2)$ with $X_1+X_2\as C_1$ and $Y_1+Y_2\as C_2$, then
$$
X_2-C_1\as -X_1 \sim -Y_1\as Y_2-C_2 \sim X_2 -C_2,
$$
which clearly implies $C_1=C_2$. More generally, one obtains the following result.

\begin{proposition}\label{th:th1-2}
Suppose that $(\mu_1,\dots,\mu_n)$ is JM and at least $n-2$ of $\mu_1,\dots,\mu_n$ have finite mean.
Then, the center of $(\mu_1,\dots,\mu_n)$ is unique.
\end{proposition}
\begin{proof}
Let $n> 2$ and let $(X_1,\dots,X_n)$ be a joint mix for $(\mu_1,\dots,\mu_n)$ with center $C$. Without loss of generality, assume that $\mu_1,\dots,\mu_{n-2}$ have finite mean. Then, $\sum_{i=1}^{n-2}X_i$ is integrable and $X_{n-1}+X_n\as C-\sum_{i=1}^{n-2}X_i$. Thus, $\E(X_{n-1}+X_n)$ is finite, so that $\E(X_{n-1}+X_n)$ only depends on on $\mu_{n-1}$ and $\mu_n$ (because of \cite{gS77}). Hence,
$$C=\sum_{i=1}^{n-2}\E(X_i)+\E(X_{n-1}+X_n)$$
is the only center of $(\mu_1,\dots,\mu_n)$. \end{proof}

Another uniqueness criterion can be obtained by increasing $n-2$ to $n-1$ but replacing in the above proposition the existence of the mean with the slightly weaker condition
\begin{gather}\label{wlln}
\lim_{x\rightarrow \infty}x\,\mu(\{y\in \R: |y|>x\})=0.
\end{gather}

\begin{proposition}\label{th:th-weak}
Suppose that $(\mu_1,\dots,\mu_n)$ is JM and at least $n-1$ of $\mu_1,\dots,\mu_n$ satisfy condition \eqref{wlln}.
Then, the center of $(\mu_1,\dots,\mu_n)$ is unique.
\end{proposition}

\begin{proof}
Let $(X_1,\dots,X_n)$ be a joint mix for $(\mu_1,\dots,\mu_n)$ with center $C$. Without loss of generality, assume that $\mu_1,\dots,\mu_{n-1}$ satisfy condition \eqref{wlln}.
Then, $\mu_n$ also satisfies condition \eqref{wlln}. In fact, for every $x\ge n\,\abs{C}$, one obtains
$$\p(\abs{X_n}>x)=\p\bigl(\abs{C-\sum_{i=1}^{n-1}X_i}>x\bigr)\le\sum_{i=1}^{n-1}\p\bigl(\abs{X_i}>x/n\bigr).$$
Hence, $x\,\p(\abs{X_n}>x)\rightarrow 0$ as $x\rightarrow\infty$.

Now take iid copies of $(X_1,\cdots,X_n)$, denoted by
$\{(X_{1,k},\cdots,X_{n,k})\}_{k=1}^\infty$.
For $m\in \N$, let $c_{i,m}=\E(X_i\id_{\{|X_i|\leq m\}})$, $i=1,\dots,n$, and  $c_m=\sum_{i=1}^{n}c_{i,m}$. By condition \eqref{wlln} and the weak law of large numbers, we have $\frac{1}{m}\sum_{k=1}^m X_{i,k}-c_{i,m}\overset{\p}\longrightarrow 0$, as $m\to \infty$, for fixed $i=1,\dots,n$.
It follows that
\begin{eqnarray*}
C-c_m&=&\sum_{i=1}^n\,\Bigl(\frac{1}{m}\sum_{k=1}^m X_{i,k}-c_{i,m}\Bigr)\overset{\p}\longrightarrow 0\quad\text{as }m\to \infty.
\end{eqnarray*}
Therefore, $C=\lim_mc_m$ is unique.
\end{proof}

Contrary to the cases $n=1$ and $n=2$, a JM $n$-tuple of distributions may have more than one center of $n\geq 3$.

Recall that the standard Cauchy distribution is the probability measure on $\R$ with density $f(x)=\frac{1}{\pi} \frac{1}{1+x^2}$ with respect to the Lebesgue measure. Let Cauchy($\sigma$) denote the distribution of $\sigma\,X$, where $\sigma>0$ and $X$ has the standard Cauchy distribution. In~\cite{CS83}, Chen and Shepp show the existence of two Cauchy($4$) random variables $U,V$ and a constant $C \neq 0$ such that $U+V+C$ is Cauchy($4\sqrt{2}$). Thus, $(-U,-V,U+V+C)$ and $(U,V,-U-V-C)$ are both joint mixes for the triplet $\bigl($Cauchy($4$), Cauchy($4$), Cauchy($4\sqrt{2})\bigr)$ with centers $C$ and $-C$, respectively. From Proposition~\ref{bv6r}, one also obtains a 3-CM probability measure with an interval of centers.

\begin{example}[A probability measure with an interval of centers]\label{ex:0}
Let $\nu=\text{Cauchy}(4)$ and $\gamma=\text{Cauchy}(4\sqrt{2})$.
By \cite{CS83}, the triplet $(\nu,\nu,\gamma)$ is JM with centers $C$ and $-C$ for some $C>0$. Take two independent joint mixes
$(X_1,X_2,X_3)$ and $(Y_1,Y_2,Y_3)$ for $(\nu,\nu,\gamma)$ with centers $C$ and $-C$, respectively. Fix $\alpha \in [0,1]$ and define $Z_i=\alpha X_i+(1-\alpha)Y_i, 1 \leq i \leq 3$.
Using characteristic functions, it is straightforward to see that $Z_1\sim Z_2 \sim\nu$ and $Z_3\sim\gamma$. Hence,
$(Z_1,Z_2,Z_3)$ is a joint mix for $(\nu,\nu,\gamma)$ with center $(2\alpha-1)C$. Since $\alpha \in [0,1]$ is arbitrary, Proposition~\ref{bv6r} implies that each point in the interval $[-C/3,C/3]$ is a 3-center of $\mu=(2\nu+\gamma)/3$. Note however that such $\mu$
does not belong to the Cauchy family of distributions.
\end{example}

The general question of whether the center of a JM $n$-tuple of distributions is \emph{always} unique was stated as an open problem in~\cite{WW11,PW15,W15}. During the writing of the present paper, we became aware of the Chen-Shepp example in~\cite{CS83} providing an early negative answer to the question.

However, the Chen-Shepp example, while implying non-uniqueness of the center, does not provide an explicit
construction for it depends on (the existence of) an orthogonal projection. Furthermore, the value of $C$ is not explicitly given and it is not clear if and how it can be computed.
We next give an example of couplings having the same marginal distributions but different sums. To the best of our knowledge, this is the first explicit construction of two joint (complete) mixes having the same marginal distributions and different centers.

\begin{example}[Center of a JM triplet is not unique] \label{ex:01}
Let $Z$ be a random variable with a geometric distribution with parameter $1/2$, that is,
$\p(Z=k)=2^{-(k+1)}$, $k \geq 0$, and let $B$ be a Bernoulli random variable with parameter $1/2$ independent of $Z$.
Let
$$X_1=X_2=2^Z,~X_3=-2^{Z+1},$$
and
$$Y_1=B2^{Z+1}+(1-B),~Y_2=(1-B)2^{Z+1}+B,~Y_3=-2^{Z+1}.$$
Then, $X_1+X_2+X_3=0$, $Y_1+Y_2+Y_3=1$ and $X_3=Y_3$. Furthermore,  $\p(X_1=1)=1/2=\p(Y_1=1)$ and
\begin{gather*}
\p(Y_1=2^k)=\p(B=1)\p(Z=k-1)=2^{-(k+1)}=\p(X_1=2^k)
\end{gather*}
for each $k\ge 1$. Similarly, $X_2\sim Y_2$. Thus, if $\nu$ denotes the distribution of $X_1=X_2$ and $\gamma$ that of $X_3$, the triplet $(\nu,\nu,\gamma)$ is JM with centers 0 and 1. From this example and Proposition~\ref{bv6r}, it also follows that the probability measure $\mu=\frac{2}{3}\nu+\frac{1}{3}\gamma$ is 3-CM with centers 0 and 1/3. In Example \ref{ex:02again} below we shall see that $0$ and $1/3$ are actually the only $3$-centers of $\mu$.
\end{example}
%

\section{The set of centers of mixable distributions}

Let $\Lambda_n(\mu)$ be the set of those $\lambda\in\Gamma(\mu,\ldots,\mu)$ such that
$$
\lambda\big(\bigl\{x\in\mathbb{R}^n:x_1+\ldots+x_n=C\bigr\}\big)=1
$$
for some $C\in\mathbb{R}$. With a slight abuse of terminology, we call $C/n$ the $n$-center of $\lambda$. Clearly, $\Lambda_n(\mu)\neq\emptyset$ if and only if $\mu$ is $n$-CM and each $\lambda\in\Lambda_n(\mu)$ is the distribution of a $n$-complete mix for $\mu$. We also denote by $\phi$ the function
$$\phi(\lambda)=\text{\upshape center}(\lambda),\quad\text{ for }\lambda\in\Lambda_n(\mu).$$

\begin{proposition}\label{th:metric}
Let $\mu$ be $n$-CM. Then, $\Lambda_n(\mu)$ is a compact metric space and the function $\phi:\Lambda_n(\mu)\rightarrow\R$
is continuous. In particular, the set $\phi(\Lambda_n(\mu))$ of $n$-centers of $\mu$ is compact. Thus, there exist $a \le b$ such that $a$ and $b$ are $n$-centers of $\mu$ but no point in $(-\infty,a)\cup (b,+\infty)$ is an $n$-center of $\mu$.
\end{proposition}
\begin{proof} Let $\mathcal{P}_n$ be the set of all probability measures on $\mathbb{R}^n$, equipped with the topology of weak convergence, and let $S(x)=\sum_{i=1}^nx_i,$ for $x\in\mathbb{R}^n$. Since $\Lambda_n(\mu)\subset\mathcal{P}_n$ and $\mathcal{P}_n$ is a Polish space, $\Lambda_n(\mu)$ is compact if and only if it is closed and tight. Fix $\lambda_k\in\Lambda_n(\mu)$ and $\lambda\in\mathcal{P}_n$ such that $\lambda_k\rightarrow\lambda$ weakly as $k\rightarrow\infty$. Then, $\lambda\in\Gamma(\mu,\ldots,\mu)$ since $\lambda_k\in\Gamma(\mu,\ldots,\mu)$ for all $k$ and the coordinate maps $x\mapsto x_i$ are continuous for all $i=1,\ldots,n$. Similarly, since $S$ is continuous, $\lambda_k\circ S^{-1}\rightarrow\lambda\circ S^{-1}$ weakly. In addition, $\lambda_k\in\Lambda_n(\mu)$ implies $\lambda_k\circ S^{-1}=\delta_{n\phi(\lambda_k)}$ for all $k$. Since $\{\delta_x:x\in\R^n\}$ is a closed subset of $\mathcal{P}_n$, it follows that $\lambda\circ S^{-1}=\delta_x$ for some $x\in\R^n$ and $n\,\phi(\lambda_k)\rightarrow x$ as $k\rightarrow\infty$. Hence, $\lambda\in\Lambda_n(\mu)$ and $\phi(\lambda)=x/n=\lim_k\phi(\lambda_k)$. This proves that $\Lambda_n(\mu)$ is closed and $\phi$ is continuous. Finally, $\Lambda_n(\mu)$ is tight since all its elements have the same one-dimensional marginals (all equal to $\mu$). Thus, $\Lambda_n(\mu)$ is a compact metric space.\end{proof}

If $\mu$ is $n$-CM with an unique center, then $a=b$ in Proposition ~\ref{th:metric}.
In case $a<b$, a natural question is: which points in the interval $[a,b]$ are centers of $\mu$ ?

A probability measure supported by the integers, such as $\mu$ in Example \ref{ex:01}, has at most finitely many centers because of Proposition~\ref{th:metric}.
This conclusion can be actually generalized to any discrete distribution.

\begin{proposition}\label{disc}
If $(\mu_1,\dots,\mu_n)$ is JM and each $\mu_i$ is discrete, then $(\mu_1,\dots,\mu_n)$ has finitely many centers.
\end{proposition}
\begin{proof}
Since $\mu_i$ is discrete, there is a finite set $A_i$ such that $0\in A_i$ and $\mu_i(A_i)>1-1/n$. Let $(X_1,\dots,X_n)$ be a joint mix for $(\mu_1,\dots,\mu_n)$ with center $C$ and $B=\bigcap_{i=1}^n\{X_i\in A_i\}$. Then,
$$C\id_{B}\as (X_1+\dots+X_n)\id_{B}\in A_1+\dots+A_n,$$
where $A_1+\dots+A_n=\bigl\{\sum_{i=1}^nx_i:x_i\in A_i, 1 \leq i \leq n\bigr\}$. Since $\p(B)>0$, it follows that $C$ belongs to the finite set $A_1+\dots+A_n$.
\end{proof}

The situation is quite different for diffuse distributions, which can have infinitely many centers; see for instance Example~\ref{ex:0}. A more interesting case is  exhibited by Theorem \ref{th:cauchy} below, where $\mu$ is the standard Cauchy and each point in $[a,b]$ is an $n$-center of $\mu$.

We next obtain two useful bounds for $a$ and $b$ in Proposition \ref{th:metric}. Define the quantile functional
$$q_\mu(t)=\inf\{x\in \R: \mu((-\infty,x])\ge t\},~~t\in (0,1).$$
For $ 0 <\alpha < \beta < 1$,  define also the average quantile functional
$$R_{[\alpha,\beta]} (\mu) = \frac{1}{\beta -\alpha } \int _{\alpha} ^{\beta} q_\mu(t)\dif t.$$
Note that, for fixed $\alpha$ and $\beta$, the map $\mu\mapsto R_{[\alpha,\beta]}(\mu)$  is continuous with respect to weak convergence.
\begin{proposition}\label{m9hj}
Suppose that  $(\mu_1,\dots,\mu_n)$ is JM with center $C$. Then, for any $\beta_1 ,\dots, \beta_n \in(0,1) $ such that $\beta: = \beta_1 +\dots+\beta_n <1 $, one obtains
\begin{equation}\label{eq:ncCM}
\sum_{i=1}^n  R_{[\beta_i ,1- \beta +\beta_i ] }(\mu_i) \le  C \le \sum_{i=1}^n  R_{[\beta- \beta_i ,1-\beta_i ] }(\mu_i).
\end{equation}
\end{proposition}

\begin{proof}
The first inequality follows from the second by noting that $-C$ is a center of $(\mu_1^*,\dots,\mu_n^*)$, where $\mu_i^*(A)=\mu_i(-A)$ for each $A\in\mathcal{B}$. Hence, we only prove the second inequality.

By applying $\alpha_i=\beta_i$, $\beta_i=1-\beta-\epsilon$
in Theorem 1 of \cite{ELW16} (with their notation $\mathrm{RVaR}_{a,b}(X)=R_{[1-a-b,1-a]}(\mu)$ for $X\sim \mu$, $a,b>0$, $a+b<1$), we obtain, for any integrable random variables $Y_1\sim \nu_1,\dots,Y_n\sim\nu_n$, with $\sum_{i=1}^{n}Y_i =Y\sim \nu$, that
$$ R_{[\epsilon,1-\beta]}(\nu)= \mathrm{RVaR}_{\beta,1-\beta-\epsilon}(Y) \le \sum_{i=1}^n \mathrm{RVaR}_{\beta_i,1-\beta-\epsilon}(Y_i) = \sum_{i=1}^n   R_{[\beta-\beta_i+\epsilon,1-\beta_i]}(\nu_i),$$
for all $\epsilon\in (0,1-\beta)$.

 Take a joint mix $(X_1,\ldots,X_n)$ for $(\mu_1,\dots,\mu_n)$ with center $C$ and a sequence $\{(X_{1,k},\ldots,X_{n,k})\}_{k=1}^\infty$ satisfying
\begin{gather*}
X_{i,k}\text{ is integrable for all }i,\,k\text{ and}
\\(X_{1,k},\ldots,X_{n,k})\overset{d}\longrightarrow (X_1,\ldots,X_n)\quad\text{as }k\rightarrow\infty
\end{gather*}
where $\overset{d}\longrightarrow$ stands for convergence in distribution. Then, $\sum_{i=1}^nX_{i,k}\overset{d}\longrightarrow\sum_{i=1}^nX_i\as C$ and  $X_{i,k}\overset{d}\longrightarrow X_i$ for all $i$. By continuity of the average quantile functional with respect to weak convergence, it follows that
\begin{gather*}
C=R_{[\epsilon,1-\beta]} (\delta_C)  \le \sum_{i=1}^n  R_{[\beta - \beta_i +\epsilon,1-\beta_i] }(\mu_i)\quad\text{for any }\epsilon\in(0,1-\beta).
\end{gather*}
Finally, by taking  $\epsilon\downarrow 0$, one obtains
$$C  \le \sum_{i=1}^n  R_{[\beta - \beta_i ,1-\beta_i] }(\mu_i),  $$
which concludes the proof.
\end{proof}

Letting $\beta_1=\ldots=\beta_n$ and $\mu_1=\ldots=\mu_n$, Proposition \ref{m9hj} has the following useful consequence.

\begin{corollary}\label{coro:bounds}
If $\mu$ is $n$-CM with center $c$, then $ a^*\le c\le b^*$, where
$$a^*=\sup_{\alpha \in (0,\frac1n )}   R_{[\alpha ,1-  (n-1) \alpha ] }(\mu)\quad\text{and}\quad b^*=\inf_{\alpha \in  (0,\frac1n  ) }   R_{[(n-1)\alpha ,1-\alpha  ] }(\mu).$$
\end{corollary}

\begin{example}[The mean inequality]\label{ar5cv7}
A remarkable consequence of Corollary \ref{coro:bounds} is the mean inequality (Proposition 2.1(7) of \cite{WW11}), arguably the most important necessary condition for complete mixability, which is also sufficient for probability measures with monotone densities.
Let $x=\lim_{\epsilon\downarrow 0}q_\mu(\epsilon)$ and $y=\lim_{\epsilon\downarrow 0}q_{\mu}(1-\epsilon)$ be the left and right end-points of $\mu$.
Assume $x$ and $y$ are finite and denote by $c$ the mean of $\mu$. If $\mu$ is $n$-CM, Corollary \ref{coro:bounds} yields $c\le b^*$. Hence, for $\alpha \in (0,1/n)$, we have that
$
c \leq R_{[(n-1)\alpha ,1-\alpha ]}(\mu),
$
that is
\begin{equation}\label{quick}
\frac{c- (1-n\alpha) R_{[(n-1)\alpha ,1-\alpha ]}(\mu)}{\alpha} \le n\,c.
\end{equation}
On the other hand,
\begin{gather*}
\lim_{\alpha \downarrow 0}\frac{1}{\alpha}\left(c-(1-n\alpha) R_{[(n-1)\alpha ,1-\alpha ]}(\mu)\right)=\lim_{\alpha \downarrow 0} \frac{1}{\alpha }\left(\int_0^1 q_\mu(t)\dif t - \int^{1-\alpha}_{(n-1)\alpha} q_\mu(t)\dif t\right)
\\=\lim_{\alpha \downarrow 0} \frac{1}{\alpha }\left(\int_{1-\alpha}^1 q_\mu(t)\dif t+\int_0^{(n-1)\alpha} q_\mu(t)\dif t\right) = y+(n-1)x.
\end{gather*}
Therefore, inequality \eqref{quick} yields
$ {y+(n-1)x} \le nc$, one side of the mean inequality in \cite{WW11} (the other follows similarly).
\end{example}

\begin{example}[Example \ref{ex:01} revisited]\label{ex:02again}
The probability measure $\mu$ defined  in Example \ref{ex:01} is 3-CM with 3-centers 0 and $1/3$. We now prove that $0$ and $1/3$ are actually the only 3-centers of $\mu$. For $n=3$,
one can compute that
$$a^*\ge \lim_{\alpha\downarrow 0}   R_{[\alpha ,1-  2 \alpha ] }(\mu)  =0 ~~\mbox{and}~~b^*\le \lim_{\alpha\downarrow 0}    R_{[2\alpha ,1-\alpha  ] }(\mu)  =2/3.$$
By Corollary \ref{coro:bounds}, it follows that $a=0$ and $b\le 2/3$.
Let $(X_1,X_2,X_3)$ be a complete mix for $\mu$. Since $\mu(\mathbb{Z})=1$, $a=0$ and $b\le 2/3$, then $X_1+X_2+X_3\in [0,2]\cap\mathbb{Z}$ a.s. Thus, to see that 0 and 1/3 are the only $3$-centers of $\mu$, it suffices to show that $\p(X_1+X_2+X_3=2)<1$.
Since $X_2+X_3 \neq 1$ a.s., then $\p(X_1+X_2+X_3= 2)\le \p(X_1\ne 1)<1$.
\end{example}

\bigskip

Another consequence of Proposition \ref{m9hj} is that a distribution with an infinite mean cannot be $n$-CM for any $n\in \N$.
The following corollary can be shown by letting $\beta_i\downarrow 0$, $i=1,\dots,n$ in \eqref{eq:ncCM}, so that the left-hand side of \eqref{eq:ncCM} goes to infinity.
\begin{corollary}\label{coro:infinite}
If $\mu_1,\dots,\mu_n$ have means $m_1,\dots,m_n\in (-\infty,\infty]$, respectively, and $m_i=\infty$ for at least one $i=1,\dots,n$,
then $(\mu_1,\dots,\mu_n)$ is not JM. In particular, a probability measure on $\mathbb{R}$ with an infinite mean is not $n$-CM for any $n\in \N$.
\end{corollary}

 \bigskip
We conclude this section by characterizing the set of $n$-centers of a probability measure based on a duality argument.

Let $\mu$ be a probability measure on $\R$. Recall that $S(x)=\sum_{i=1}^nx_i$ for $x\in \RR^n$ and write $\{S=nc\}$ to denote the set $\{x \in \mathbb{R}^n: S(x)=nc\}$. By definition, a real number $c$ is an $n$-center of $\mu$ if and only if $M(c)=1$, where
\begin{equation*}
M(c)=\sup\bigl\{ \lambda(S=nc): \lambda\in \Gamma(\mu,\dots,\mu)\bigr\}.
\end{equation*}
Based on Theorem 5 of~\cite{lR81} and Remark~2 in~\cite{GR81}, $M(c)$ has the dual representation
\begin{equation}
M(c)=n \inf \left\{ \int g \;d\mu : g \in \mathcal{D}(c)\right\},\label{eq:dual1}
\end{equation}
where $\mathcal{D}(c)$ denotes the class of bounded, Borel-measurable functions
$g:\RR\to\RR$ such that
$
\sum_{i=1}^{n}g(x_i)\geq \mathbf{1}_{\{S=nc\}}(x)
$
for all $x \in \RR^n$. The value of \eqref{eq:dual1} is not easy to compute in general.
However, restricting to a subset of $\mathcal D(c)$ (as done for instance in~\cite{EP06b}) leads to an upper bound for $M(c)$.

We consider the following class of piecewise-linear functions defined,
for $t<c$, as
\begin{equation*}
g_t(x)=\begin{cases}
0 &\text{if $x<t$,}\\
\frac{x-t}{n(c-t)}&\text{if $t \leq x \leq nc-(n-1)t$,}\\
1 &\text{otherwise.}
\end{cases}
\end{equation*}
Since $g_t \in \mathcal D(c)$ for all $t<c$, we obtain
\begin{equation*}
M(c) \leq D(c):=\inf_{t<c} \left\{ n \int g_t \;d\mu \right\}= \inf_{t < c} \left\{ \frac{\int_{t}^{nc-(n-1)t} \F(x) \; dx
}{c-t}\right\}
\end{equation*}
where $\F(x)=\mu((x,\infty))$. If $D(c) <1$, then $c$ is not an $n$-center of $\mu$.
Therefore, for $c$ to be an $n$-center of $\mu$, it is necessary that
$$ \frac{\int_{t}^{nc-(n-1)t} \F(x) \; dx
}{c-t}\ge 1\quad\text{for all $t<c$}.$$
The above inequality is another necessary condition for the center of $n$-CM probability measures, in addition to that of Corollary \ref{coro:bounds}. These two necessary conditions are not equivalent in general.


\section{The Cauchy distribution}

From now on, we let $\mu=\text{Cauchy}(1),$ the standard Cauchy distribution.
It is shown in \cite{RU02b} that $\mu$ is $n$-CM with center 0, for each $n\ge 2$, as it is symmetric and unimodal.
In this section, we characterize the set of $n$-centers of $\mu$.
We start by observing that such set is a closed interval contained in $[-\frac{\log (n-1)}{\pi},\frac{\log (n-1)}{\pi}] $.

\begin{example}\label{ex:cauchy}
As in Proposition ~\ref{th:metric}, let $a$ and $b$ be the minimum and the maximum of the set of $n$-centers of $\mu$. Let $X=(X_1,\dots,X_n)$ and $Y=(Y_1,\dots,Y_n)$ be two independent complete mixes for $\mu$ such that
$$
\sum_{i=1}^{n}X_i \as na \quad \text{ and } \quad \sum_{i=1}^{n}Y_i \as nb.
$$
Fix $\alpha \in [0,1]$ and define $Z_i=\alpha X_i+(1-\alpha)Y_i$ for $1 \leq i \leq n$.
Then, $Z_i\sim\mu$ for each $i$ and $\sum_{i=1}^n Z_i=n(\alpha a+(1-\alpha)b)$ so that $\alpha a+(1-\alpha)b$ is a center of $\mu$. Hence, $\phi(\Lambda_n(\mu))=[a,b]$, namely, each point in $[a,b]$ is a center of $\mu$.

Next, on noting that $q_\mu(t)=\tan(\pi(t-1/2))$, one obtains
\begin{align*}
  R_{[(n-1)\alpha ,1-\alpha  ] }(\mu) &=   \frac{1}{1-n\alpha}\int_{(n-1)\alpha}^{1-\alpha} \tan(\pi(t-1/2))\dif t\\
  &=   \frac{1}{1-n\alpha}\frac{1}{\pi}\log\left(\frac{\sin(\pi(n-1)\alpha)}{\sin(\pi\alpha)}\right)\quad\text{for }\alpha \in (0,1/n).
  \end{align*}
By Corollary \ref{coro:bounds},
  $$b\le b^*\le \lim_{\alpha\downarrow 0}    R_{[(n-1)\alpha ,1-\alpha  ] }(\mu)= \frac{\log (n-1)}{\pi}.$$
Since $a=-b$ (for $\mu$ is symmetric) one also obtains $a \ge -\frac{\log (n-1)}{\pi}$.
\end{example}
\bigskip
Example \ref{ex:cauchy} says that $\phi(\Lambda_n(\mu))\subset[-\frac{\log (n-1)}{\pi},\frac{\log (n-1)}{\pi}]$.
Our main result is that this inclusion is an equality.
\bigskip
\begin{theorem}\label{th:cauchy}
For every $n\ge 2$, the set of $n$-centers of the standard Cauchy distribution is the interval
$$\left[-\frac{\log (n-1)}{\pi},\frac{\log (n-1)}{\pi}\right].$$
\end{theorem}

\vspace{0.2cm}

The rest of this section is devoted to the proof of Theorem \ref{th:cauchy}.

For each $c\in\R$, let $\mathcal M_n(c)$ denote the collection of $n$-CM probability measures with center $c$. We first need two lemmas of possible independent interest. The first states that $\mathcal M_n(c)$ is closed under arbitrary mixtures, generalizing Theorem 3.2 of \cite{PWW12}.

\begin{lemma}\label{ucciard77}
Let $(T,\mathcal{E},Q)$ be any probability space and, for each $t\in T$, let $\nu_t\in\mathcal M_n(c)$. Suppose that $t\mapsto\nu_t(B)$ is a $\mathcal{E}$-measurable map, for each $B\in\mathcal{B}$, and define
\begin{gather*}
\nu(B)=\int\nu_t(B)\,Q(\dif t).
\end{gather*}
Then, $\nu\in\mathcal M_n(c)$.
\end{lemma}

\begin{proof}
Let $\mathcal{R}$ be the field on $\mathbb{R}^n$ generated by the  measurable rectangles $B_1\times\ldots\times B_n$, where $B_i\in\mathcal{B}$ for all $i$, and let $\gamma:\mathcal{R}\rightarrow\R$ be any map. By Theorem 6 of \cite{dR96}, $\gamma$ is a $\sigma$-additive probability on $\mathcal{R}$ provided it is a {\it finitely additive} probability and $A\mapsto\gamma\bigl\{x\in\mathbb{R}^n:x_i\in A\bigr\}$ is a $\sigma$-additive probability on $\mathcal{B}$ for all $i$.

Let $H=\{x\in\mathbb{R}^n:x_1+\ldots+x_n=nc\}$. For each $t\in T$, since $\nu_t\in\mathcal M_n(c)$, there is $\lambda_t\in\Gamma(\nu_t,\ldots,\nu_t)$ such that $\lambda_t(H)=1$. Define
\begin{gather*}
\lambda^*(B)=\int\lambda_t(B)\,Q^*(dt)\quad\text{for each }B\in\mathcal{R},
\end{gather*}
where $Q^*$ is a finitely additive extension of $Q$ to the power set of $T$. Then, $\lambda^*$ is a finitely additive probability on $\mathcal{R}$ and
\begin{gather*}
\lambda^*\bigl\{x\in\mathbb{R}^n:x_i\in A\bigr\}
=\int\lambda_t\bigl\{x\in\mathbb{R}^n:x_i\in A\bigr\}\,Q^*(dt)\\=\int\nu_t(A)\,Q^*(dt)
=\int\nu_t(A)\,Q(dt)=\nu(A)
\end{gather*}
for all $A\in\mathcal{B}$ and all $i=1,\ldots,n$, where the third equality holds because $t\mapsto\nu_t(A)$ is $\mathcal{E}$-measurable. Hence, $\lambda^*$ is $\sigma$-additive on $\mathcal{R}$. Let $\lambda$ be the only $\sigma$-additive extension of $\lambda^*$ to the Borel $\sigma$-field of $\mathbb{R}^n$. Since $\lambda\in\Gamma(\nu,\ldots,\nu)$, to conclude the proof it suffices to see that $\lambda(H)=1$. In fact, since $H^c$ is open, it is a countable union of open rectangles, that is, $H^c=\cup_kG_k$ with $G_k\in\mathcal{R}$ for all $k$. Since $\lambda_t(G_k)\le\lambda_t(H^c)=0$ for all $t\in T$, one obtains
\begin{gather*}
\lambda(G_k)=\lambda^*(G_k)=\int\lambda_t(G_k)\,Q^*(dt)=0\quad\text{for all }k.
\end{gather*}
Therefore $\lambda(H)=1$, for $H^c$ is a countable union of $\lambda$-null sets.
\end{proof}

The second lemma, which slightly generalizes Theorem 2.4 of \cite{WW11}, provides conditions for a certain probability measure to be $n$-CM.

\begin{lemma}\label{us45b}
Let $\alpha\in [0,1]$, $x<y$, and $\nu$ a probability measure on $[x,y]$ which admits a non-increasing density (with respect to the Lebesgue measure). Then, $\alpha\delta_x+(1-\alpha)\nu$ is $n$-CM if and only if
\begin{gather}\label{ns5r}
\alpha\le 1-\frac{y-x}{n(q-x)}\quad\quad\text{where }\,q=\int z\,\nu(\dif z)\,\text{ is the mean of }\nu.
\end{gather}
\end{lemma}

\begin{proof}
Let $\gamma=\alpha\,\delta_x+(1-\alpha)\,\nu$. As noted in Example \ref{ar5cv7}, a necessary condition for complete mixability is the mean inequality in~\cite{WW11}, and condition \eqref{ns5r} is precisely the mean inequality for $\gamma$. Thus, \eqref{ns5r} holds if $\gamma$ is $n$-CM. Conversely, suppose \eqref{ns5r} holds. Let $\gamma_\epsilon=\alpha\,\U_{[x,\,x+\epsilon]}+(1-\alpha)\,\nu$, where $\epsilon\in (0,\,y-x)$ and $\U_I$ stands for the uniform distribution on the interval $I$. By \eqref{ns5r}, $\gamma_\epsilon$ satisfies the mean inequality. By Corollary 2.9 of \cite{WW11}, since $\gamma_\epsilon$ has non-increasing density and meets the mean inequality, $\gamma_\epsilon$ is $n$-CM. Further, as $\epsilon\rightarrow 0$, the mean of $\gamma_\epsilon$ converges to $\alpha x+(1-\alpha)q$ and $\gamma_\epsilon\rightarrow\gamma$ weakly. Thus, $$\gamma\in\mathcal M_n\bigl(\alpha x+(1-\alpha)q\bigr)$$
because of Theorem 3.1 of \cite{PWW12}.
\end{proof}

We are now ready to prove Theorem \ref{th:cauchy}.

\begin{proof}[\bf Proof of Theorem \ref{th:cauchy}]

Recall that $\mu\in \mathcal M_n(0)$ (see~\cite{RU02b}) for all $n\ge 2$.  Since the case $n=2$ is trivial, we assume $n\ge3$.  Fix $c\in (0, \frac{\log (n-1)}{\pi}].$
By Example \ref{ex:cauchy}, it suffices to show that $\mu\in\mathcal M_n(c)$.
In turn, by Lemma~\ref{ucciard77}, it suffices to prove that $\mu$ can be written as
\begin{equation} \label{eq:mut} \mu=\int\mu_t\,Q( \dif t)\end{equation}
where $\mu_t\in\mathcal M_n(c)$ for all $t>0$ and $Q$ is a probability measure on $(0,\infty)$.

Let $f(x)=\frac{1}{\pi} \frac{1}{1+x^2}$, $x\in\R$, be the standard Cauchy density and $f^{-1}$ the function on $[0, 1/\pi]$ given by
$$f^{-1}(0)=\infty,\quad f^{-1}(x)=\sqrt{\frac{1}{\pi\,x}-1},\quad\quad\text{for }x\in (0, 1/\pi].$$
Also, let $h:(0,\infty)\to \R$ be a $C^1$ function such that, for each $t>0$:
\begin{enumerate}[(I)]
\item   $\int_{c-t}^{c+(n-1)t}  (x-c)\,\{f(x) - h(t)\} _+ \dif x=0,$
\item $0\le h(t)\le  f(c+t)$,
\item $h'(t)\le 0$.
\end{enumerate}
The existence of such $h$ will be verified at the end of the proof. For the moment, we assume that $h$ exists.

For $t>0$, let  $\nu_t$ be the finite measure on $\mathcal{B}$ with density
$$f_t(x) = \bigl\{f(x) - h(t)\bigr\}_+\,\mathbf{1}_{[c-t,c+(n-1)t]}(x).$$
Since $f_t(x)\le f(x)$ and $f(x)=\lim_{t\rightarrow\infty}f_t(x)$ for all $t>0$ and $x\in\R$, one obtains
\begin{gather*}
\lim_{t\rightarrow\infty}\nu_t(B)=\lim_{t\rightarrow\infty}\int_Bf_t(x)\,dx=\int_Bf(x)\,dx=\mu(B)
\end{gather*}
for each $B\in\mathcal{B}$. Note also that $\lim_{t\to 0} \nu_t(B)=0$, for
\begin{gather*}
\limsup_{t\to 0} \nu_t(B)\le\limsup_{t\to 0} \nu_t(\R)\le\limsup_{t\to 0}\int_{c-t}^{c+(n-1)t}f(x)\,dx=0.
\end{gather*}

Denote $$ K_1(t) =  f(c-t) - h(t) ,~~ K_2(t) = (n-1)\,\bigl\{f(c+(n-1)t) -h(t)\bigr\}_+, $$   $$ K_3(t) = \min \{ c+(n-1)t ,\, f^{-1} (h(t)) \}, ~~   K_4(t) = -h'(t)\,\bigl( K_3(t) -(c-t)\bigr).$$
Since $f(c+t) < f(c-t)$, $t>0$, conditions (II)-(III) imply $K_i(t)\ge 0$ for all $i$ with $K_1(t)>0$ and $K_3(t)\ge c+t$. Hence, for each $t>0$, one can define
\begin{gather*}\mu_t = \frac{K_1(t) \delta _{ c-t } +  K_2(t) \delta_{c+(n-1)t} + K_4(t)\,\U_{[c-t, K_3(t) ]}}{K_1(t)+K_2(t)+K_4(t)}
\end{gather*}
where $\U_{[x,y]}$ denotes the uniform distribution on the interval $[x,y]$. Such $\mu_t$ are the probability measures that we use in \eqref{eq:mut}.

Next, fix $y\in\R$ and define
\begin{gather*}
I_y=(-\infty,y),\quad g(t)=\nu_t(I_y),\quad Q_t=K_1(t)+K_2(t)+K_4(t)\quad\text{for each }t>0.
\end{gather*}
Then, $g$ is continuous and satisfies, by direct calculation,
\begin{align*}g'(t)=\frac{\dif }{\dif t} \nu_t(I_y)= Q_t\,\mu_t(I_y)\quad\text{for all }t\in T,\end{align*}
where $T\subset (0,\infty)$ is a co-finite set (possibly depending on $y$). Since $g'$ is locally integrable (with respect to the Lebesgue measure) it follows that
\begin{gather}\label{eq:cauchy3}
\nu_t(I_y)=\lim_{\epsilon\rightarrow 0}\{g(t)-g(\epsilon)\}=\lim_{\epsilon\rightarrow 0}\int_\epsilon^tg'(s)\,ds=\int_0^tQ_s\,\mu_s(I_y)\,ds
\end{gather}
for all $t>0$. In particular,
\begin{gather*}
\int_0^\infty Q_s\,ds=\lim_{t\rightarrow \infty}\int_0^t Q_s\,\Bigl(\lim_{y\rightarrow \infty}\mu_s(I_y)\Bigr)\,ds
=\lim_{t\rightarrow \infty}\lim_{y\rightarrow \infty}\int_0^t Q_s\,\mu_s(I_y)\,ds\\=\lim_{t\rightarrow \infty}\lim_{y\rightarrow \infty}\nu_t(I_y)
=\lim_{t\rightarrow \infty}\nu_t\bigl(\R)=\mu(\R)=1.
\end{gather*}
Let  $Q$ be the probability measure on $(0,\infty)$ such that $Q((0,t])=\int_0^t Q_s\dif s$ for all $t>0$. Then, condition \eqref{eq:cauchy3} yields
\begin{gather*} \int_0^\infty\mu_t(I_y)\,Q(\dif t)=\lim_{s\rightarrow\infty}\int_0^s Q_t\,\mu_t(I_y)\,dt=\lim_{s\rightarrow\infty}\nu_s(I_y)=\mu(I_y)\quad\text{for all }y\in\R.
\end{gather*}
Therefore,
\begin{gather*} \int_0^\infty\mu_t(B)\,Q(\dif t)=\mu(B)\quad\text{for each }B\in \mathcal B.\end{gather*}

To prove $\mu_t\in \mathcal M_n(c)$, it is fundamental to note that $\mu_t$ has mean $c$. Define in fact
\begin{align*}
\phi(t)=\int_{c-t}^{c+(n-1)t}  (x-c)\bigl\{f(x) - h(t)\bigr\} _+ \dif x
&=\int_{c-t}^{K_3(t)}  (x-c)\bigl(f(x) - h(t)\bigr) \dif x  \\
&=\int_{c-t}^{K_3(t)}  (x-c)f(x) \dif x -h(t) \int_{c-t}^{K_3(t)}  (x-c) \dif x.
\end{align*}
By condition (I), $\phi(t)=0$ for all $t>0$. Computing $\phi'(t)$, one obtains
\begin{align*}
0=\phi'(t)&=-t\,K_1(t) + (n-1)\,t\,K_2(t)-h'(t)\int_{c-t}^{K_3(t)} (x-c)\dif x\\&=Q_t \int (x-c)\mu_t (\dif x).
\end{align*}
Therefore, $\mu_t$ has mean $c$ for all $t>0$.

Having noted this fact, fix $t>0$ and define
\begin{gather*}\mu_t ^{(1)} = \frac{n-1}{ n }  \delta _{ c-t }  + \frac{1}{ n } \delta _{ c+(n-1)t }.
\end{gather*}
Such $\mu_t ^{(1)}$ has mean $c$ and is $n$-CM by Example \ref{bh9ik7}. Hence, $\mu_t ^{(1)}\in\mathcal M_n(c)$. If $K_4(t)=0$, then $\mu_t$ is a convex combination of $\delta _{ c-t }$ and $\delta_{c+(n-1)t}$. Since $\mu_t$ has mean $c$, it follows that $\mu_t=\mu_t ^{(1)}\in\mathcal M_n(c)$.

Suppose now that $K_4(t)>0$. Since $K_3(t)\ge c+t$, the mean of $\U_{[c-t, K_3(t) ]}$ is not less than $c$. Since $\mu_t$ has mean $c$, it follows that
$$\frac{K_1(t) \delta _{ c-t } +  K_2(t) \delta_{c+(n-1)t}}{K_1(t)+K_2(t)}$$
has a mean smaller than or equal to $c$, namely, $K_1(t) \ge  (n-1) K_2(t) $. By this fact and $K_4(t)>0$, one can define
\begin{gather*}\mu_t ^{(2)} = \frac{K_1(t) - (n-1)K_2(t)  }{ K_1(t) - (n-1) K_2(t) + K_4(t)}  \delta _{ c-t }  +\frac{K_4(t) }{ K_1(t) - (n-1) K_2(t) + K_4(t) } \U_{[ c-t, K_3(t) ]}.\end{gather*}
Since $\mu_t$ and $\mu_t ^{(1)}$ have mean $c$ and
$$\mu_t = \frac{n K_2(t)}{Q_t}\,\mu_t ^{(1)} + \frac{K_1(t) - (n-1) K_2(t) + K_4(t) }{Q_t}\,\mu_t ^{(2)},$$
then $\mu_t ^{(2)}$ has mean $c$ as well. By Lemma \ref{us45b}, $\mu_t ^{(2)}$ is $n$-CM (condition \eqref{ns5r} follows from $\mu_t ^{(2)}$ having mean $c$). Therefore, $\mu_t ^{(2)}\in\mathcal M_n(c)$. Finally, since $\mu_t ^{(i)}\in\mathcal M_n(c)$ for $i=1,2$ and $\mathcal M_n(c)$ is convex (by Lemma \ref{ucciard77}), one obtains $\mu_t\in\mathcal M_n(c)$.

To conclude the proof, it remains only to prove that a $C^1$-function $h$ satisfying conditions (I)-(II)-(III) actually exists. Define
\begin{equation*}
A(t,y)= \int _{c-t}^{c+(n-1)t} (x-c)  \bigl\{f(x) - y\bigr\}_+ \dif x\quad\text{for all } t > 0\text{ and }y\in\R.
\end{equation*}
Then, $A$ is a  $C^1$-function on $(0,\infty)\times\R$ and
\begin{gather*}
\frac{\partial A}{\partial y} (t,y)=  \int _{c-t}^{c+(n-1)t} - (x-c)  \mathbf{1}_{\{f>y\}}(x)\,dx,
\end{gather*}
where $\{f>y\}$ denotes the set $\{x\in\R:f(x)>y\}$. Fix $t>0$. If $y<f(c+t)$, there is $u\in (c+t,\,c+(n-1)t]$ such that $f(x)>y$ for every $x\in [c-t,u]$. Thus,
\begin{gather*}
\frac{\partial A}{\partial y} (t,y)\le -\int _{c-t}^{u} (x-c)\,dx=-\int _{c+t}^{u} (x-c)\,dx<0.
\end{gather*}
Hence, the map $y\mapsto A(t,y)$ is continuous, strictly decreasing on $(-\infty,\,f(c+t)]$, and
\begin{align*}
\lim_{y\rightarrow -\infty}A(t,y)=\infty, \quad A\bigl(t,\,f(c+t)\bigr)&=\int _{c-t}^{c+t} (x-c)\bigl(f(x) - f(c+t)\bigr)\,dx\\&=\int _{-t}^{t} x f(c+x) \,dx<0.
\end{align*}
It follows that, for each $t>0$, there exists a unique number $h(t)$ satisfying $h(t)<f(c+t)$ and $A\bigl(t,\,h(t)\bigr)=0$. It remains to see that $h\ge 0$, $h$ is $C^1$ and $h'\le 0$.

We begin with $h\ge 0$. Since $A(t,y)>A(t,0)$ whenever $y<0$, it suffices to see that $A(t,0)\ge 0$. Define
\begin{gather*}m(t)=\frac{1}{t}\,\frac{\dif}{\dif t}A(t,0)=(n-1)^2f\bigl(c+(n-1)t\bigr)-f(c-t)\quad\text{for }t>0.
\end{gather*}
Then,
\begin{gather*}m'(t)=(n-1)^3f'\bigl(c+(n-1)t\bigr)+f'(c-t)=(n-1)^3f'\bigl(c+(n-1)t\bigr)-f'(t-c).
\end{gather*}
Observe now that, since $\sqrt{1/f}$ is convex (see Remark~\ref{convex}),
$$\frac{\dif \sqrt{1/f(x)}}{\dif x }= -\frac{1}2\frac{f'(x)}{f^{3/2}(x)}$$
is an increasing function of $x\in \R$. Therefore,
$$\frac{f'(c+(n-1)t)}{\bigl(f(c+(n-1)t)\bigr)^{3/2} } \le  \frac{f'(t-c)}{\bigl(f(t-c)\bigr)^{3/2}}$$
and, by rearranging terms,
\begin{gather}\label{qc57jm90}\frac{f'(t-c)}{(n-1)^3 f'(c+(n-1)t)}\le\Bigl(\frac{f(t-c)}{(n-1)^2 f(c+(n-1)t)}\Bigr)^{3/2}.\end{gather}
If $m(t)\ge 0$, then $(n-1)^2 f(c+(n-1)t)\ge f(c-t)=f(t-c)$, so that the right-hand member of \eqref{qc57jm90} is bounded above by 1. Hence, $m(t)\ge 0$ implies $m'(t)\le 0$. Thanks to this fact and $\lim_{t\downarrow 0}m(t)>0$, one concludes that there is some $t_0>0$ (possibly $t_0=\infty$) such that $m(t)\ge 0$ for $t\le t_0$ and $m(t)\le 0$ for $t\ge t_0$. Since $m(t)$ and $\frac{\dif}{\dif t}A(t,0)$ have the same sign and
\begin{align*} \lim_{t\downarrow 0} A(t,0)=0,\quad\lim_{t\rightarrow \infty} A(t,0) = \lim_{t\rightarrow \infty} \int_{c-t}^{c+(n-1)t} x   f(x) \dif x - c = \frac{\log(n-1)}{\pi} - c \ge 0,\end{align*}
one finally obtains $A(t,0)\ge 0$. Therefore, $h(t)\ge 0$ for all $t>0$.

To prove $h$ is $C^1$, recall that $h(t)$ is the only real number such that $h(t)<f(c+t)$ and $A(t,h(t))=0$. Also, $h(t)<f(c+t)$ implies $\frac{\partial A}{\partial y} (t,h(t))<0$. Thus, $h$ is $C^1$ because of the implicit function theorem.

We finally prove $h'\le 0$. Since $\frac{\partial A}{\partial y} (t,h(t))<0$ and
\begin{gather*}
h'(t)=-\frac{\frac{\partial A}{\partial t} (t,h(t))}{\frac{\partial A}{\partial y} (t,h(t))},
\end{gather*}
it suffices to show that $\frac{\partial A}{\partial t} (t,h(t))\le 0$ for $t>0$. Since
\begin{gather*}
\frac{1}{t}\,\frac{\partial A}{\partial t}(t,h(t))  = (n-1)^2 \bigl\{f(c+(n-1)t) - h(t)\bigr\}_+ - \bigl\{ f(c-t)-h(t) \bigr\}_+,
\end{gather*}
it can be assumed $f(c+(n-1)t)>h(t)$. In this case,
\begin{gather*}
\frac{1}{t}\,\frac{\partial A}{\partial t}(t,h(t))  = (n-1)^2\,f(c+(n-1)t)-f(c-t) - n(n-2)h(t)=m(t)-n(n-2)h(t).
\end{gather*}
Recall now that $m(t)\ge 0$ for $t\le t_0$ and $m(t)\le 0$ for $t\ge t_0$ where $t_0\in (0,\infty]$. Hence, $\frac{\partial A}{\partial t}(t,h(t))\le 0$ for $t\ge t_0$. If $t\in (0,t_0)$, since $t\,m(t)=\frac{\dif}{\dif t}A(t,0)$ and $m$ is decreasing on $(0,t_0)$, then
\begin{gather*}
A(t,0)=\int_0^ts\,m(s)\,ds\ge m(t)\,\int_0^ts\,ds=t^2m(t)/2.
\end{gather*}
On the other hand, it is easily seen that
\begin{gather*}
\frac{\partial A}{\partial y} (t,y)\ge-\int _{c-t}^{c+(n-1)t}(x-c)\,dx= -n(n-2)t^2/2.
\end{gather*}
Therefore,
\begin{gather*}
-A(t,0)=A(t,h(t))-A(t,0)=\int_0^{h(t)}\frac{\partial A}{\partial y} (t,y)\,dy\ge -h(t)\,n(n-2)t^2/2.
\end{gather*}
It follows that $m(t)\le n(n-2)h(t)$, and again $\frac{\partial A}{\partial t}(t,h(t))\le 0$. To summarize, $h'(t)\le 0$ for all $t>0$, namely, $h$ satisfies conditions (I)-(II)-(III). This concludes the proof.
\end{proof}



\begin{remark}\label{convex}
The proof of Theorem \ref{th:cauchy} is valid for other probability measures in addition to the Cauchy distribution.
 Fix in fact a probability measure $\nu$ on $\R$ which admits a density $g$ with respect to the Lebesgue measure. If $g$ is strictly positive, differentiable, symmetric and strictly unimodal (that is, $g'>0$ on $(-\infty,0)$ and $g'<0$ on $(0,\,\infty)$), and if $\sqrt{1/g}$ is a convex function, then any real number $q$ satisfying
\begin{gather*}
\abs{q}\le \liminf_{t\rightarrow \infty} \int_{t}^{(n-1)t} x   g(x) \dif x,
\end{gather*}
is an $n$-center of $\nu$. This follows from replacing $\mu$ and $c$ by $\nu$ and $|q|$, respectively, in the proof of Theorem \ref{th:cauchy}.
%
Therefore, taking an arbitrary differentiable, symmetric, strictly positive and strictly convex function $\phi$, by defining a density $g=\frac{\beta}{\phi^2}$ for some normalizing constant $\beta>0$ one always finds a distribution $\nu$ fulfilling
the requirements of the proof of Theorem~\ref{th:cauchy}.
However, apart from the Cauchy distribution, we do not know of any natural example of such $\nu$. This is due to the convexity of $\sqrt{1/g}$, which is a quite restrictive requirement.
For instance, suppose that $g$ has the power form  $g(x)=\frac{\beta}{1+|x|^\alpha}$ for some constants $\alpha,\beta>0$. In this case, $\alpha=2$ leads to the standard Cauchy distribution, $\alpha>2$ implies that the mean of $\nu$ is finite (so that 0 is the unique center), and $\alpha<2$ implies that $\sqrt{1/g}$ is not convex.
\end{remark}


\bibliographystyle{amsplain}
\bibliography{mybib}

%
%
%
%
%
%
%


\end{document}